\documentclass{amsart}
\usepackage{amssymb}
\usepackage{pb-diagram}

\newtheorem{theorem}{Theorem}[section]

\newtheorem{prop}[theorem]{Proposition}

\newtheorem{claim}[theorem]{Claim}
\newtheorem{fact}[theorem]{Fact}

\newtheorem{lemma}[theorem]{Lemma}

\theoremstyle{definition}
\newtheorem{defn}[theorem]{Definition}

\theoremstyle{remark}

\newtheorem{remark}[theorem]{Remark}

\newcommand{\WM}{\widetilde{\cal M}}

\newcommand{\la}{\langle}
\newcommand{\ra}{\rangle}

\newcommand{\sub}{\subseteq}

\newcommand{\dcl}{\operatorname{dcl}}

\newcommand{\ldim}{\ensuremath{\textup{ldim}}}

\newcommand{\bb}[1]{\ensuremath{\mathbb{#1}}}

\newcommand{\cal}[1]{\ensuremath{\mathcal{#1}}}
\newcommand{\Cal}[1]{\ensuremath{\mathcal{#1}}}

\newcommand{\rarr}{\ensuremath{\rightarrow}}

\newcommand{\res}{\ensuremath{\upharpoonright}}

\newcommand{\sm}{\setminus}

\title[Product cones]
{Product cones in dense pairs}


\date{\today}

\begin{document}

\author {Pantelis  E. Eleftheriou}

\address{Department of Mathematics and Statistics, University of Konstanz, Box 216, 78457 Konstanz, Germany}

\email{panteleimon.eleftheriou@uni-konstanz.de}

\thanks{Research supported by a Research Grant from the German Research Foundation (DFG) and a Zukunftskolleg Research Fellowship.}

\begin{abstract} Let $\cal M=\la M, <, +, \dots\ra$ be an o-minimal expansion of an ordered group, and $P\sub M$ a dense set such that certain tameness conditions hold. We introduce the notion of  a \emph{product cone} in $\widetilde{\cal M}=\la \cal M, P\ra$, and prove: if $\cal M$ expands a real closed field, then $\WM$ admits a product cone decomposition. If $\cal M$ is linear, then it does not. In particular, we settle a question from \cite{egh}.



\end{abstract}

\date{\today}

 \maketitle

\section{Introduction}

Tame expansions $\widetilde{\cal M}=\la \cal M, P\ra$ of an o-minimal structure $\cal M$ by a set $P\sub M$ have received lots of attention in recent literature (\cite{bz, beg, bh,dms1,vdd-dense,dg,gh, ms}). One important category is when every definable open set is already definable in $\cal M$.  Dense pairs and expansions of \cal M by a dense independent set or by a dense multiplicative group with the Mann Property are of this sort. 
In \cite{egh}, all these examples were put under a common perspective and a cone decomposition theorem was proved for their definable sets and functions. This theorem provided an analogue of the cell decomposition theorem for o-minimal structures in this context, and was inspired by the cone decomposition theorem established for semi-bounded o-minimal structures  (\cite{ed-str, el-sbdgroups, pet-str}). The central notion is that of a \emph{cone}, and, as its definition in \cite{egh} appeared to be quite technical, in \cite[Question 5.14]{egh}, we asked whether it can be simplified in two specific ways. In this paper we refute both ways in general, showing that the definition in \cite{egh} is optimal, but prove that if $\cal M$ expands a real closed field, then  a \emph{product cone} decomposition theorem does hold.

In Section \ref{sec-prelim}, we provide all necessary background and definitions. For now, let us only point out  the difference between product cones and  cones, and state our  main theorem. Let $\cal M=\la M, <, +, \dots\ra$ be an o-minimal expansion of an ordered group in the language $\cal L$, and $\widetilde{\cal M}=\la \cal M, P\ra$  an expansion of $\cal M$ by a set $P\sub M$ such that certain tameness conditions hold (those are listed in Section \ref{sec-prelim}). For example, $\widetilde{\cal M}$ can be a dense pair (\cite{vdd-dense}), or $P$ can be a dense independent set (\cite{dms2}) or a multiplicative group with the Mann Property (\cite{dg}). By `definable' we mean `definable in $\WM$, and by $\cal L$-definable we mean `definable in \cal M'.
The notion of a \emph{small} set is given in Definition \ref{def-small} below, and it is equivalent to the classical notion of being $P$-internal from geometric stability theory (\cite[Lemma 3.11 and Corollary 3.12]{egh}). A \emph{supercone} generalizes the notion of being co-small in an interval (Definition \ref{def-supercone}). Now, and roughly speaking, a cone is then defined as a set of the form
$$h\left(\bigcup_{g\in S} \{g\}\times J_g\right),$$
where $h$ is an $\cal L$-definable continuous map with each $h(g, -)$  injective, $S\sub M^m$ is a small set, and $\{J_g\}_{g\in S}$ is a definable family of supercones. In Definition \ref{def-cone} below, we call a cone a \emph{product cone} if we can replace the above family  $\{J_g\}_{g\in S}$ by a product $S\times J$. That is, $C$ has the form
$$h(S\times J),$$
with $h$ and $S$ as above and $J$ a supercone.
Let us say that $\WM$ \emph{admits a product cone decomposition} if every definable set is a finite union of product cones. Our main theorem below asserts whether $\WM$ admits a product cone decomposition or not based solely  on assumptions on $\cal M$.




\begin{theorem}\label{main1}   $ $
\begin{enumerate}
\item If $\cal M$ is linear, the $\WM$ does not admit a product cone decomposition.
\item If $\cal M$ expands a real closed field, then $\WM$ admits a product cone decomposition theorem.
\end{enumerate}
\end{theorem}
Theorem \ref{main1}(1), in particular, refutes \cite[Question 5.14(2)]{egh}. \cite[Question 5.14(1)]{egh}  further asked whether one can define a supercone as a product of  co-small sets in intervals, and still obtain a cone decomposition theorem. In Proposition \ref{main2} we also refute that question in general, by constructing a counterexample whenever $\cal M$ expands a real closed field.

\begin{remark} Theorem \ref{main1} deals with the two main categories of o-minimal structures; namely,  $\cal M$ is linear or it expands a real closed field. In the `intermediate', semi-bounded case (\cite{el-sbdgroups}), where $\cal M$ defines a field on a bounded interval but not on the whole of $M$, the answer to \cite[Question 5.14]{egh} is rather unclear.  Indeed, in the presence of two different notions of cones in this setting, the semi-bounded cones (from \cite{el-sbdgroups}) and the current ones,  the methods in Sections \ref{sec-linear} and \ref{sec-field} do not seem to apply and a new approach is needed.
\end{remark}



\subsection*{Notation} The topological closure of a set $X\sub M^n$ is denoted by $cl(X).$ 
Given any subset $X\subseteq M^m \times M^n$ and $a\in M^n$, we write $X_a$ for
\[
\{ b \in M^m \ : \ (b,a) \in X\}.
\]
If $m\le n$, then $\pi_m:M^n\to M^m$ denotes the projection onto the first $m$ coordinates. We write $\pi$ for $\pi_{n-1}$, unless stated otherwise.
A family $\cal J=\{J_g\}_{g\in S}$ of sets is called definable if $\bigcup_{g\in S}\{g\}\times J_g$ is definable.
We often identify $\cal J$ with $\bigcup_{g\in S}\{g\}\times J_g$.


$ $\\
\noindent\textbf{Acknowledgments.} I thank Philipp Hieronymi for several discussions on this topic.


\section{Preliminaries}\label{sec-prelim}

In this section we lay out all  necessary background and terminology. Most of it is extracted from \cite[Section 2]{egh}, where the reader is referred to for an extensive account. 
We fix an o-minimal theory $T$ expanding the theory of ordered abelian groups with a distinguished positive element $1$. We denote by $\Cal L$ the language of $T$ and by $\Cal L(P)$ the language $\Cal L$ augmented by a unary predicate symbol $P$. Let $\widetilde{T}$ be an $\Cal L(P)$-theory extending $T$. If $\cal M=\la M, <, +, \dots\ra\models T$, then $\widetilde{\cal M}=\la \cal M, P\ra$ denotes an expansion of \cal M that models $\widetilde{T}$. By `$A$-definable' we mean `definable in $\widetilde{\cal M}$ with parameters from $A$'. By `$\cal L_A$-definable' we mean `definable in \cal M with parameters from $A$'. We omit the index $A$ if we do not want to specify the parameters.
For a subset $A\subseteq M$, we write $\dcl(A)$ for the definable closure of $A$ in $\Cal M$, and for an $\cal L$-definable set $X\sub  M^n$, we write $\dim(X)$  for the corresponding pregeometric dimension. The following definition is taken essentially from \cite{dg}.


\begin{defn}\label{def-small}
 Let $X\sub M^n$ be a definable set. We call $X$ \emph{large} if there is some $m$ and an $\cal L$-definable function $f:M^{nm}\to M$ such that $f(X^m)$ contains an open interval in $M$. We call $X$ \emph{small} if it is not large.
\end{defn}

Consider the following properties.\vskip.2cm

\noindent\textbf{Tameness Conditions (\cite{egh}):}
\begin{itemize}

\item [(I)]  $P$ is small.

\item [(II)] Every $A$-definable set $X\sub M^n$ is a boolean combination of sets of the form
\[
\{x\in M^n: \exists z\in P^m \varphi(x, z)\},
\]
 where $\varphi(x, z)$ is an $\cal L_A$-formula.

\item [(III)] (Open definable sets are $\cal L$-definable) For every parameter set $A$ such that $A\setminus P$ is $\dcl$-independent over $P$, and for every $A$-definable set $V \subset M^s$, its topological closure $cl(V)\subseteq M^{s}$ is $\cal L_A$-definable.
\end{itemize}

\noindent\textbf{From now on, we assume that  every model $\WM\models \widetilde T$  satisfies Conditions (I)-(III) above. We fix a sufficiently saturated model $\widetilde{\cal M}=\la \cal M, P\ra\models \widetilde T$.}\vskip.2cm

We next turn to define the central notions of this paper. As mentioned in the introduction, the notion of a cone is based on that of a supercone, which in its turn generalizes the notion of being co-small in an interval. Both notions, supercones and cones, are unions of specific families of sets, which not only are definable, but they are so in a very uniform way.

\begin{defn}[\cite{egh}]\label{def-supercone}
A \emph{supercone} $J\sub M^k$, $k\ge 0$, is defined recursively as follows:
\begin{itemize}
\item $M^{0}=\{0\}$ is a supercone.

\item A definable set $J\sub M^{n+1}$ is a supercone if $\pi(J)\sub M^n$ is a supercone and there are  $\cal L$-definable continuous $h_1, h_2: M^n\to M\cup \{\pm\infty\}$ with $h_1<h_2$, such that for every $a\in \pi(J)$, $J_a$ is contained in $(h_1(a), h_2(a))$ and it is co-small in it.
\end{itemize}
Abusing terminology, we call a supercone  \emph{$A$-definable} if it is an $A$-definable set and its closure is $\Cal L_A$-definable.
\end{defn}

Note that, for $k>0$, the interior $U$ of $cl(J)$ is an open cell.




Recall that in our notation we identify a family $\cal J=\{J_g\}_{g\in S}$ with $\bigcup_{g\in S}\{g\}\times J_g$. In particular, $cl(\cal J)$ and $\pi_n(\cal J)$ denote the closure and a projection of that set, respectively.

\begin{defn}[Uniform families of supercones \cite{egh}]\label{def-uniform}
Let $\Cal J = \bigcup_{g\in S} \{g\}\times J_g\sub M^{m+k}$ be a definable family of supercones. We call $\Cal J$ \emph{uniform} if there is a cell $V\sub M^{m+k}$ containing $\cal J$, such that for every $g\in S$ and  $0<j\le k$,
$$cl(\pi_{m+j}(\cal J)_g)=cl(\pi_{m+j}(V)_g).$$
We call such a $V$ a \emph{shell} for $\cal J$. Abusing terminology, we call \cal J  \emph{$A$-definable}, if it is an $A$-definable  family of sets and  has an $\cal L_A$-definable shell.
\end{defn}

In particular, if $\cal J$ is uniform, then so is each projection $\pi_{m+j}(\cal J)$. Moreover, if $V$ is a shell for $\cal J$, then $\pi_{m+j}(V)$ is a shell for $\pi_{m+j}(\cal J)$.
Observe also that if $V$ is a shell for $\cal J$, then for every $x\in \pi_{m+k-1}(\cal J)$, $\cal J_x$ is co-small in $V_x$.\\

A shell for $\cal J$ need not be unique. Whenever we say that $\cal J$ is a uniform family of supercones with shell $V$, we just mean that $V$ is one of the shells for $\cal J$.

\begin{defn}[Cones \cite{egh} and product cones]\label{def-cone}
A set $C\sub M^n$ is a \emph{$k$-cone}, $k\ge 0$,  if there are a definable small $S\sub M^m$, a uniform family $\cal J=\{J_g\}_{g\in S}$ of supercones in $M^k$, and an $\cal L$-definable continuous function $h:V \subseteq M^{m+k}\to M^n$, where $V$ is a shell for $\cal J$, such that
\begin{enumerate}
\item $C=h(\cal J)$, and
\item for every $g\in S$, $h(g, -): V_g\sub M^k\to M^n$ is injective.
\end{enumerate}
We call $C$ a \emph{$k$-product cone} if, moreover, $\cal J=S\times J$, for some supercone $J\sub M^k$.
A \emph{(product) cone} is a $k$-(product) cone for some $k$. 
Abusing terminology, we call a cone $h(\cal J)$ \emph{$A$-definable} if $h$ is $\Cal L_A$-definable and  $\Cal J$ is  $A$-definable.
\end{defn}

The cone decomposition theorem below (Fact \ref{conedec}) is a statement about definable sets and functions. The notion of a `well-behaved' function in this setting is given next.

\begin{defn}[Fiber $\cal L$-definable maps \cite{egh}]\label{defn:fiber}
Let $C=h(\cal J)\sub M^n$ be a $k$-cone with $\cal J\sub M^{m+k}$, and $f: D\to M$ a definable function with $C\sub D$. We say that $f$ is \emph{fiber $\cal L$-definable with respect to $C$} if there is an $\Cal L$-definable continuous function $F: V\sub M^{m+k}\to M$, where $V$ is a shell containing $\cal J$, such that
\begin{itemize}
\item $(f\circ h)(x)=F(x)$, for all $x\in \cal J$.
\end{itemize}
We call $f$ \emph{fiber $\cal L_A$-definable with respect to $C$} if $F$ is $\cal L_A$-definable. 
\end{defn}

As remarked in \cite[Remark 4.5(4)]{egh}, the terminology is justified by the fact that, if $f$ is fiber $\cal L_A$-definable with respect to $C=h(\cal J)$, then  for every $g\in \pi(\cal J)$, $f$ agrees on $h(g, J_g)$ with an $\cal L_{Ag}$-definable map; namely $F\circ h(g, -)^{-1}$. Moreover, the notion of being fiber $\cal L$-definable with respect to a cone $C=h(\cal J)$, depends on $h$ and $\cal J$ (\cite[Example 4.6]{egh}). However, it is immediate from the definition that if $f$ is fiber $\cal L_A$-definable with respect to a cone $C=h(\cal J)$, and $h(\cal J')\sub h(\cal J)$ is another cone (but with the same $h$), then $f$ is also fiber $\cal L_A$-definable with respect to it.


We are now ready to state the cone decomposition theorem from \cite{egh}.
 \begin{fact}[Cone decomposition theorem {\cite[Theorem 5.1]{egh}}]\label{conedec}$ $
\begin{enumerate}
  \item \emph{Let $X\sub M^n$ be an $A$-definable set. Then $X$ is a finite union of $A$-definable cones.}
  \item \emph{Let $f:X\to M$ be an $A$-definable function. Then there is a finite collection $\mathcal{C}$ of   $A$-definable cones, whose union is $X$ and such that $f$ is fiber $\cal L_A$-definable with respect to each cone in $\cal{C}$.}
\end{enumerate}
 \end{fact}

Another important notion from \cite{egh} is that of `large dimension', which we recall next. The proof of Theorem \ref{main1}(2) runs by induction on large dimension.


\begin{defn}[Large dimension \cite{egh}]
Let $X\sub M^n$ be definable. If $X\ne \emptyset$, the \emph{large dimension} of $X$ is the maximum $k\in \bb N$ such that $X$ contains a $k$-cone. The large dimension of the empty set is defined to be $-\infty$. We denote the large dimension of $X$ by  $\ldim(X)$.
\end{defn}

\begin{remark} The tameness conditions that we assume in this paper guarantee that the notion of large dimension is well-defined; namely, the above maximum $k$ always exists (\cite[Section 4.3]{egh}). In fact, everything that follows only uses the following two assumptions (instead of the tameness conditions): (a) Fact \ref{conedec} and (b) the notion of large dimension is well-defined.
\end{remark}

\section{Product cone decompositions}\label{sec-main}

In this section we prove Theorem \ref{main1}.  

\subsection{The linear case}\label{sec-linear} The following definition is taken from \cite{lp}.

\begin{defn}[\cite{lp}]
Let  $\cal{N}=\langle N, +, <, 0, \dots\rangle$ be an o-minimal expansion of an ordered group. A function $f:A\sub N^n\rarr N$ is called \emph{affine}, if for every $x,y, x+t, y+t\in B$,
\begin{equation}
  f(x+t)-f(x)=f(y+t)-f(y).\label{eq-linear}
\end{equation}
We call \cal N \emph{linear} if for every definable $f:A\sub N^n\rarr N$, there is a partition of $A$ into finitely many definable sets $B$, such that each $f_{\res B}$ is affine.
\end{defn}

The typical example of a linear o-minimal structure is that of an ordered vector space $\mathcal{V}=\langle V, <, +, 0, \{d\}_{d\in D}\rangle$ over an ordered division ring $D$. In general, if $\cal N$ is linear, then there exists a reduct \cal S of such $\cal V$, such that ${\mathcal S} \equiv \cal N$ (\cite{lp}). Using this description, it is not hard to see that every affine function has a continuous extension to the closure of its domain. 

Assume now that our fixed structure $\cal M$ is linear.

\begin{lemma}\label{H}
Let $h:[a, b]\times [c,d]\to M$ be an $\cal L$-definable continuous function, such that for every $t\in (a, b)$, $h(t, -):[c, d]\to M$ is strictly increasing. Then
$$h(b, d)-h(b, c)>0.$$
\end{lemma}
\begin{proof}
Let $\cal W$ be a cell decomposition of $[a, b]\times [c,d]$ such that for every $W\in \cal W$, $h_{\res W}$ is affine. Since $d-c>0$, there must be some $W=(f,g)_I\in \cal W$, where $I$ is an interval with $\sup I=b$, and $r\in I$, such that the map $\delta(t):=g(t)-f(t)$ is increasing on $[r, b)$. We claim that for every $t\in (r, b)$,
$$h(t, g(t))-h(t, f(t))\ge h(r, g(r)) -h(r, f(r)).$$
Indeed, there is $k\ge 0$, such that
$$h(t, f(t) +\delta(t)) - h(t, f(t))=h(t, f(t) +\delta(r)+k) - h(t, f(t))=$$
$$h(t, f(t) +\delta(r)+k) +h(t, f(t) +\delta(r)) - h(t, f(t) +\delta(r)) + h(t, f(t))\ge$$
$$\ge h(t, f(t) +\delta(r))- h(t, f(t))=h(r, f(r) +\delta(r)) -h(r, f(r)),$$
where the inequality is because $h(t, -)$ is increasing, and the last equality because $h$ is affine on $W$.
We conclude that
\begin{align}
  h(b, d)-h(b,c) &= \lim_{t\to b}\big(h(t, d)-h(t,c)\big)\notag\\
  &\ge \lim_{t\to b} \big(h(t, g(t))-h(t, f(t))\big)\notag
  \\&\ge h(r, g(r)) -h(r, f(r))\notag\\
  &>0,\notag
\end{align}
where the first and last inequalities are because $h(t, -)$ and $h(r, -)$ are  strictly increasing.
\end{proof}




\noindent\textbf{Counterexample to product cone decomposition.}
Let $S\sub M$ be a small set such that $0$ is in the interior of its closure (by translating $P$ to the origin, such an $S$ exists).
Let
$$X=\bigcup_{a\in S^{>0}}\{a\}\times (0, a).$$

\begin{claim}
  $X$ is not a finite union of product cones.
\end{claim}

\begin{proof}
First of all, $X$ cannot contain any $k$-cones for $k>1$, since $\ldim(X)=1$, by \cite[Lemma 4.24 and 4.27]{egh}.
Now let $H(T\times J)$ be an $1$-product cone contained in $X$, with $H=(H_1, H_2):Z\sub M^{l+1}\to M^2$, such that the origin is in its closure. Since $H$ is $\cal L$-definable and continuous, and for each $g\in T$, $H_2(g, -)$ is injective, we may assume that the latter is always strictly increasing. 
By \cite[Lemma  5.10]{egh} applied to $J$, $f(-)=\pi_1 H(g, -)$ and $S$, we have
\begin{itemize}
   \item for every $g\in T$, there is $a\in S$, such that $H(g, J)\sub \{a\}\times (0, a)$.
       \end{itemize}
 By continuity of $H$, it follows that
 \begin{itemize}
 \item for every $g\in cl(T)\cap\pi(Z)$, there is $a\in M$, such that $$H(g, cl(J))\sub \{a\}\times [0, a].$$
 \end{itemize}
Let $f:\pi(Z)\to M$ be the $\cal L$-definable map given by
$$f(g)=\pi_1 (H(g, cl(J))).$$
Since the origin is in the closure of $H(T\times J)$, there must be an affine $\gamma:(a,  b)\to cl(T)\cap \pi(Z)$ with $\lim_{t\to b} f(\gamma(t))=0$. Now the map
$$H_2(\gamma(-),-):(a,b)\times (c,d)\to M$$
is affine and hence has a continuous extension $h$ to $[a, b]\times [c,d]$. By definition of $X$,
$$h(b, c)=h(b, d)=0.$$
But, by  Lemma \ref{H},
$$h(b, d)-h(b, c)>0,$$
a contradiction.
\end{proof}


\subsection{The field case}\label{sec-field}


We now assume that \cal M expands an ordered field. The main idea behind the proof in this case is as follows. By Fact \ref{conedec}, it suffices to write every cone as a finite union of product cones. We illustrate the case of an $1$-cone $C=h(\cal J)$, for some $\cal J=\{J_g\}_{g\in S}$. \\

\noindent\textbf{Step I} (Lemma \ref{Lfact}). Replace $\cal J$ by a cone $\cal J'=\{J'_g\}_{g\in S}$, such that for some fixed interval $I$, each  $J'_g$ is contained in $I$ and it is co-small in it. Here we use the field structure of $\cal M$, so this step would fail in the linear case.\\

\noindent\textbf{Step II} (Lemma \ref{field1}). By \cite[Lemma 4.25]{egh}, the intersection $J=\bigcap_{g\in S}  J'_g$ is co-small in $I$. Moreover, if we let $L=S\times J$, then, by \cite[Lemma 4.29]{egh}, we obtain that the large dimension of $\cal J\sm L$ is $0$.\\

\noindent\textbf{Step III} (Theorem \ref{rcf}). Use Steps I and II and induction on large dimension. Here, the inductive hypothesis is only applied to sets of large dimension $0$. In general, $\ldim(\cal J\sm L)<\ldim (\cal J)$. \\


To achieve Step I, we first need to  make an observation and fix some notation. Using the field multiplication, one can define an $\cal L_\emptyset$-definable continuous $f:M^3\to M$, such that for every $b, c\in M$,
$$f(b, c, -): (b, c)\to (0, 1)$$
is a bijection. Similarly, there are $\cal L_\emptyset$-definable continuous maps $f_{1}, f_2:M^2\to M$, such that for every $b, c\in M$, the maps
$$f_{1}(b, -): (b, +\infty)\to (0, 1)$$
and
$$f_{2}(c, -): (-\infty, c)\to (0, 1)$$
are bijections. To give all these  maps a uniform notation, we write $f(b, +\infty, x)$ for $f_1(b,x)$, and $f(-\infty,c, x)$ for $f_2(c, x)$ and . We fix this $f$ for the next proof. Observe that if $J\sub (b,c)$ is co-small in $(b,c)$, for $b,c \in M\cup\{\pm \infty\}$, then $f(b,c, J)$ is co-small in $(0, 1)$.




\begin{lemma}\label{Lfact}
Let $\cal J\sub M^{m+k}$ be an $A$-definable  uniform family of supercones, with  shell $Z\sub M^{m+k}$. Then there are
 \begin{itemize}
   \item an $A$-definable  uniform family $\cal J'=\{J'_g\}_{g\in S}$ of supercones $J'_g\sub M^k$, with a shell $\pi(Z)\times (0,1)^k$,
   \item and an $\cal L_A$-definable continuous and injective map $F:Z\to M^{m+k}$, 
 such that $$F(\cal J)=\cal J'.$$
 \end{itemize}
\end{lemma}

\begin{proof}


For every $g\in \pi(\cal J)$, since $J_g$ is a supercone, it follows that $Z_g$ is an open cell. Hence, for every $0<j\le k$, there are $\cal L_A$-definable continuous  maps $h_1^j, h_2^j: \pi_{l+j-1}(Z)\to M$ such that
$$\pi_{m+j}(Z)=(h_1^j, h_2^j)_{\pi_{m+j-1}(Z)}.$$
We define
$$F=(F_1, \dots, F_{m+k}):Z \to M^{m+k},$$ as follows. Let $I=(0,1)$ and $f$ be the map fixed above. Let $(g, t) \in Z\sub M^{m+k}$\smallskip

\noindent If $1\le i\le m$,
$$F_i(g,t)=g_i.$$

\noindent If $i=m+j$, with  $0<j\le k$,
$$F_{m+j}(g,t)=f(h_1^j(g, t_1, \dots, t_{j-1}), h_2^j(g, t_1, \dots, t_{j-1}), t_j).$$
Clearly, $F$ is injective, $\cal L_A$-definable and continuous. Let $$\cal J'=F(\cal J).$$
That is, $\cal J'=\{J'_g\}_{g\in S}$, where for every $g\in S$, $J'_g=F(g, J_g)$.
It is not hard to check, by induction on $m$, that for every $0<m\le k$, $\pi_{m+j}(\cal J')$ is an $A$-definable uniform family of supercones with  shell $F(Z)=\pi(Z)\times I^m$. 
\end{proof}


\begin{lemma}\label{field1}
Let $\cal J=\bigcup_{g\in S} \{g\}\times J_g\sub M^{m+k}$ be an $A$-definable uniform family of supercones in $M^k$ with shell $Z$. Suppose that $Z=\pi(Z)\times I^k$, where $I=(0,1)$.
Then $\cal J$ is a disjoint union
$$(S\times J) \cup Y,$$
where $S\times J$ is an $A$-definable uniform family of supercones with shell $Z$, 
and $Y$ is an $A$-definable set of large dimension $<k$. 
\end{lemma}
\begin{proof}
By induction on $k$. For $k=0$, the statement is trivial. We assume the statement holds for $k$, and prove it for $k+1$. Let $\pi:M^{m+k+1}\to M^{m+k}$ be the projection onto the first $m+k$ coordinates. Since $\pi(\cal J)$ is also an $A$-definable uniform family of supercones with shell $\pi(Z)$, by  inductive hypothesis we can write $\pi(\cal J)$ as a disjoint union
$$\pi(\cal J)= (S\times T) \cup Y,$$
where $T\sub M^{k}$ is an $A$-definable supercone with $cl(T)=cl(I^k)$, and $Y$ is an $A$-definable set of large dimension $<k$. By \cite[Corollary 5.5]{egh}, the set
$\bigcup_{t\in Y}\{t\}\times \cal J_t\sub \cal J$
has large dimension $<k+1$, and hence we only need to bring its complement $X$ in $\cal J$ into the desired form. We have
$$X=\bigcup_{t\in S\times T}  \{t\}\times \cal J_t.$$
Define, for every $a\in T$,
$$K_a=\bigcap_{g\in S} \cal J_{g, a}.$$
Since each $\cal J_{g, a}$ is co-small in $I$, by \cite[Lemma 4.25]{egh} $K_a$ is co-small in $I$. Hence, the set
$$L=\bigcup_{a\in T}\{a\}\times K_a$$
is a supercone in $M^{k+1}$. Since $cl(T)=cl(I^k)$ and for each $a\in T$, $cl(K_a)=cl(I)$, it follows that $cl(L)=cl(I^{k+1})$. In particular, $Z$ is a shell for $S\times L$. Since $S\times L\sub X$, it remains to prove that $\ldim(X\sm (S\times L))<k+1$. We have
$$X\sm (S\times L)= \bigcup_{(g, a)\in S \times T}\{(g, a)\}\times (\cal J_{g,a}\sm K_a).$$
But $\cal J_{g,a}\sm K_a$ is small, and hence,  by \cite[Lemma 4.29]{egh}, the above set has large dimension $=\ldim(S\times T)= k$.
\end{proof}

We can now conclude the main theorem of the paper.
\begin{theorem}[Product cone decomposition in the field case]\label{rcf}
Let $X\sub M^n$ be an $A$-definable set. Then
\begin{enumerate}
  \item $X$ is a finite union of $A$-definable product cones.

\item If $f:X\to M$ is an $A$-definable function, then there is a finite collection $\mathcal{C}$ of  $A$-definable product cones, whose union is $X$ and such that $f$ is fiber $\cal L_A$-definable with respect to each cone in $\cal{C}$.
\end{enumerate}
\end{theorem}
\begin{proof}

(1). We do induction on the large dimension of $X$. By Fact \ref{conedec}, we may assume that $X$ is a $k$-cone. Every $0$-cone is clearly a product cone. Now let $k>0$. By induction, it suffices to write $X$ as a union of an $A$-definable product cone and an $A$-definable set of large dimension $<k$.  Let $X=h(\cal J)$ be as in Definition \ref{def-cone}, and $Z\sub M^{m+k}$ a shell for $\cal J$. \vskip.2cm

\noindent\textbf{Claim.} \emph{We can write $X$  as a $k$-cone $h'(\cal J')$, such that for every $g\in \pi(\cal J')$, $cl(\cal J')_g=(0,1)^k$.}

\begin{proof}[Proof of Claim] Let $\cal J'$ and $F:Z\to M^{m+k}$ be as
in Lemma \ref{Lfact}, and define $h'=h\circ F^{-1}: F(Z)\to M^n$. Then
$$h(\cal J) = h F^{-1}(F(\cal J))=h'(\cal J')$$
is as required.
\end{proof}

By the claim, we may assume that for every $g\in S$, $cl(\cal J)_g=(0,1)^k$. By Lemma \ref{field1}, we have $\cal J=(S\times J) \cup Y$, where $J\sub M^k$ is an $A$-definable supercone, and $\ldim Y<k$. Thus $h(\cal J)= h(S\times J) \cup h(Y)$ has been written in the desired form.\\

\noindent (2). By Fact \ref{conedec}, we may assume that $X$ is a $k$-cone and that $f$ is fiber $\cal L_A$-definable with respect to it. So let again $X=h(\cal J)$ with shell $Z\sub M^{m+k}$, and in addition, $\tau: Z\sub M^{m+k}\to M$ with $\cal J\sub Z$, be $\cal L_A$-definable so that for every $x\in \cal J$, $$(f\circ h)(x)=\tau (x).$$
By induction on large dimension, it suffices to show that $X$ is the union of  a product cone $C$ and a set of large dimension $<k$, such that $f$ is fiber $\cal L_A$-definable with respect to $C$. Let $X=h'(\cal J')$ be as in Claim of Item (1) and $F:Z\to M^{m+k}$ as in its proof. So $h'= h\circ F^{-1}:F(Z)\to M^n$. Define $\tau' : F(Z)\to M^n$ as $\tau'=\tau \circ F^{-1}$. We then have, for every $x'\in \cal J'$,
$$fh'(x')=fh'F(x)=fh(x)=\tau(x)=\tau F^{-1}(x)=\tau'(x),$$
witnessing that $f$ is fiber $\cal L_A$-definable with respect to $h'(\cal J')$.

Therefore,  we may replace $h$ by $h$ and $\cal J$ by $\cal J'$. Now, as in the proof of Item (1), we can write $h(\cal J)$ as the union of a product cone $h(S\times J)$ and a set of large dimension $<k$. By the remarks following Definition \ref{defn:fiber}, $f$ is also fiber $\cal L$-definable with respect to $h(S\times J)$.
\end{proof}


\begin{remark} The above proof yields that in cases where we have disjoint unions in Fact \ref{conedec} (such as in \cite[Theorem 5.12]{egh}), so do we in Theorem \ref{rcf}. 
\end{remark}


\section{Refined supercones}\label{5.13(1)}

In this section we refute \cite[Question 5.14 (1)]{egh}. The question asked whether the Structure Theorem holds if we  strengthen the notion of a supercone as follows.


\begin{defn} A supercone $\cal J$ in $M^k$ is called \emph{refined} if it is of the form  $$\cal J=J_1\times \dots \times J_k,$$
where each $J_i$ is a supercone in $M$. Let us call a ($k$-)cone $C=h(\cal J)$ a \emph{($k$-)refined cone} if $\cal J$ is refined.
\end{defn}

Our result is the following. 


\begin{prop}\label{main2}
Assume $\cal M$ expands a real closed field. 
Then there is a supercone in $M^2$ which  contains no $2$-refined cone. In particular, it is not a finite union of refined cones.
\end{prop}

\begin{proof} 
\footnote{The proof  is based on an idea suggested by P. Hieronymi.}
The `in particular' clause follows from \cite[Corollaries 4.26 and 4.27]{egh}.
Now, for every $a\in M$, let
$$J_a=M\sm (P+aP)$$
and define $\cal J=\bigcup_{a\in M} \{a\}\times J_a$. Towards a contradiction, assume that  $\cal J$ contains a $2$-cone. That is, there are supercones $J_1, J_2\sub M$, an open cell $U\sub M^2$ with $cl(J_1\times J_2)=cl(U)$, and an $\cal L$-definable continuous and injective map $f:U\to M^2$, such that $C=f(J_1\times J_2)\sub \cal J$. We write $X=f(U)$, and for each $a\in M$, $X_a\sub M$ for the fiber of $X$ above $a$. Suppose $C$ is $A$-definable.

By saturation, there is $a\in M$ which is $\dcl$-independent over $A \cup P$, and further $g, h\in P$ which are $\dcl$-independent over $a$. So $$\dim(g,h, a)=3.$$
By assumption, there are $(p, q)\in U\sm (J_1\times J_2),$ such that
$$f(p,q)=(a, g+ha).$$
Observe  that $a\in \dcl(p,q)$. Also, one of $p,q$ must be in $\dcl(A P)$. Indeed,  we have $p\not\in J_1$ or $q\not \in J_2$. If, say, the former holds,  then $p\in\pi(U)\sm J_1$. Since the last set is $A$-definable and small, we obtain by \cite[Lemma 3.11]{egh}, that $p\in \dcl(AP)$.  

We may now assume that $p\in\dcl(AP)$. If we write $f=(f_1, f_2)$, we obtain
\begin{equation}
  f_2(p, q)=g+h f_1(p, q).\label{g}
\end{equation}
Since $a$ is $\dcl$-independent over $A\cup P$, there must be an open interval $I\sub M$ of $p$, such that for every $x\in I$,
$$f_2(x,q)=g+h f_1(x,q).$$
Viewing both sides of the equation as functions in the variable $f_1(x, q)$, and taking their derivatives with respect to it, we obtain:
$$\frac{\partial f_2(x,q)}{\partial f_1(x,q)}= f_1(x,q)+ h.$$
Evaluated at $p$, the last equality gives $h\in \dcl(p,q)$. By (\ref{g}), also $g\in \dcl(p,q)$.  All together, we have proved that $g, h, a\in\dcl(p,q)$. It follows that
$$\dim(g,h, a) \le \dim(p,q)\le 2,$$
a contradiction.
\end{proof}

\end{document}